\documentclass[11pt,a4paper,reqno,tbtags]{amsart}
\usepackage{amsmath}
\usepackage{amsfonts}
\usepackage{amssymb}
\usepackage{bbm}
\usepackage{latexsym}
\usepackage{graphicx}
\usepackage[utf8]{inputenc}
\usepackage{mathrsfs}

\usepackage{blindtext}
\usepackage{floatrow}
\usepackage{multicol}

\usepackage{epstopdf}
\usepackage{hyperref}
\usepackage{url}
\usepackage{mdwlist}
\newfloatcommand{capbtabbox}{table}[][\FBwidth]

\numberwithin{equation}{section}

\newtheorem{proposition}{Proposition}[section]

\newtheorem{theorem}[proposition]{Theorem}

\newtheorem{lemma}[proposition]{Lemma}

\theoremstyle{definition}
\newtheorem{remark}[proposition]{Remark}

\theoremstyle{remark}

\renewcommand\P{\mathbb{P}}
\newcommand\E{\mathbb{E}}

\newcommand{\w}{\circ}
\renewcommand{\b}{\bullet}

\newcommand{\indic}[1]{\mathbf{1}_{\{#1\}}}

\newcommand\R{\mathbb{R}}

\newcommand{\bea}{\begin{eqnarray}}
\newcommand{\eea}{\end{eqnarray}}

\def\void{}
\def\labelmark{}

\newenvironment{formula}[1]{\def\labelname{#1}
\ifx\void\labelname\def\junk{\begin{displaymath}}
\else\def\junk{\begin{equation}\label{\labelname}}\fi\junk}%
{\ifx\void\labelname\def\junk{\end{displaymath}}
\else\def\junk{\end{equation}}\fi\junk\labelmark\def\labelname{}}

{\ifx\void\labelname\def\junk{\end{array}\end{displaymath}}
\else\def\junk{\end{array}\right.\end{equation}}
\fi\junk\labelmark\def\labelname{}\def\junk{}
}

\newcommand{\beq}{\begin{formula}}
\newcommand{\eeq}{\end{formula}}
\newcommand{\beqv}{\begin{formula}{}}

\newenvironment{romenumerate}[1][0pt]{
\addtolength{\leftmargini}{#1}\begin{enumerate}
 }{\end{enumerate}}

\newcounter{oldenumi}
{\setcounter{oldenumi}{\value{enumi}}
\begin{romenumerate} \setcounter{enumi}{\value{oldenumi}}}
{\end{romenumerate}}

\begingroup
  \divide\count255 by 60
  \multiply\count255 by -60
  \advance\count255 by \time
  \ifnum \count255 < 10 \xdef\klockan{\the\count1.0\the\count255}
  \else\xdef\klockan{\the\count1.\the\count255}\fi
\endgroup

\newcommand\marginal[1]{\marginpar{\raggedright\parindent=0pt\tiny #1}}
\newcommand\REM[1]{{\raggedright\texttt{[#1]}\par\marginal{XXX}}}

\def\rompar(#1){\textup(#1\textup)}    

\def\xexp(#1){e^{#1}}

\newcommand{\tend}{\longrightarrow}

\newcommand\asto{\overset{\mathrm{a.s.}}{\tend}}

\newcommand{\NN}{\mathbb N}
\newcommand{\ZZ}{\mathbb Z}
\newcommand{\PP}{\mathbb P}
\newcommand{\EE}{\mathbb E}
\newcommand{\RR}{\mathbb R}
\renewcommand{\SS}{\mathbb S}

\newcommand{\se}{\subseteq}
\renewcommand{\S}{\Sigma}
\renewcommand{\a}{\alpha}

\newcounter{CC}
\newcounter{cc}

\newcommand{\be}{\begin{equation}}
\newcommand{\ee}{\end{equation}}

\newcommand{\uq}{\mathbf{Q}}
\newcommand{\peel}{\mathrm{Peel}}
\newcommand{\Ex}{\mathcal{E}}
\newcommand{\Sw}{\mathcal{R}}
\newcommand{\crit}{\mathrm{c}}
\newcommand{\oo}{\infty}
\newcommand{\iid}{\textsc{iid}}
\newcommand{\uihpq}{\textsc{uihpq}}
\allowdisplaybreaks

\begin{document}
\title[Site percolation in random quadrangulations]
{On site percolation in random quadrangulations of the half-plane}

\date{\today}   

\author{Jakob E. Bj\"ornberg}
\address{JEB: Department of Mathematical Sciences, 
Chalmers and Gothenburg University, 
412 96 G\"oteborg, Sweden}
\thanks{The research of \textsc{jeb} is supported by the Knut and
  Alice Wallenberg Foundation.}

\author{Sigurdur \"Orn Stef\'ansson} 
 \address{S\"OS: Division of Mathematics, The Science Institute,
University of Iceland, Dunhaga 3
IS-107 Reykjavik, Iceland.}

\maketitle

\begin {abstract}
We study site percolation on uniform quadrangulations of the
upper half plane.
The main contribution is a method for 
applying Angel's peeling process, in particular
for analyzing
an evolving boundary condition during the peeling.  
Our method lets  us  obtain rigorous and explicit 
upper and lower bounds on
the percolation threshold $p_\crit$, and thus show
in particular that $0.5511\leq p_\crit\leq 0.5581$. 
The method can
be extended to site percolation on other 
half-planar maps with the domain Markov property.
\end {abstract}

\section{Introduction}
\label{intro_sec}

Recent years have seen much 
progress on, and growing interest in, models
of statistical physics defined on random maps embedded on surfaces.
This includes work on 
percolation~\cite{angel:2003,angel:2013,menard-nolin,ray:2013}, 
simple random 
walk~\cite{angel:2014,benjamini-curien,benjamini-schramm,bjo-stef,ggn}
and also for example $O(n)$-loop 
models~\cite{bbg1,bbg2}
(see also the recent~\cite{cur-legall}).
Angel's seminal paper~\cite{angel:2003} 
provided key tools for understanding 
both the local limits of planar maps themselves, 
as well as percolation on these local
limits, by establishing a form of spatial Markovian property.
This property is
encapsulated in the so-called peeling process, which 
allows one to discover the map in a step-wise manner.
The peeling process takes a
particularly simple form for \emph{half}-planar maps, 
where it has been used to analyze not only site-percolation on
triangulations, but also for example edge- 
and face percolation on triangulations and 
quadrangulations~\cite{angel:2013}.  The method
has also been used for edge- and site-percolation on uniform planar 
maps as well as edge-percolation on uniform planar
quadrangulations~\cite{menard-nolin}. 

This paper applies the peeling process to analyze 
site percolation on
uniform infinite half-planar quadrangulations
(\uihpq).  
Recall that site-percolation on a graph is defined by assigning one of
two colours to each of the vertices, here called black and white,
independently for different vertices.  
We denote the probability of colouring a vertex white
by $p$ and we assume this to be the same for all vertices.
One is interested in whether or not the white connected cluster
containing some particular site (in our case a fixed point $0$ on the
boundary) is almost surely finite, or if it is infinite with positive
probability.  
Roughly speaking, one may use the peeling process
to discover (part of)
the map itself alongside the
outer boundary of the percolation cluster.  One keeps track of an
`active' part of the boundary of the cluster, whose size after $n$ steps 
is denoted $S_n$, defined in such a way that the cluster is finite if
and only if $S_n=0$ for some $n$.  

In the well-understood case of site
percolation on random half-planar 
\emph{triangulations}, the process $S_n$ is
a random walk.  Its i.i.d.\ increments  have a distribution which
can be found explicitly as a function of 
the percolation parameter $p$.  By computing the drift
of this random walk one arrives at the critical value $p_\crit$,
which is $1/2$ for uniform triangulations 
(the value of $p_\crit$ is known also for
more general triangulations satisfying the domain Markov 
property~\cite{ray:2013}).  Angel and Curien also obtained certain critical
exponents by using this method~\cite{angel:2013}.

A key aspect of the method is to identify an
appropriate `invariant boundary
condition' which describes the colours of the vertices on the 
boundary on one particular
side (usually the left) of
the evolving percolation cluster. 
In the case of triangulations this boundary condition is 
deterministic (all black) and thus particularly
simple.  For edge-percolation 
on triangulations and quadrangulations,
Angel and Curien used the invariance of the `free'
boundary condition.
The fact that 
the invariant boundary condition is no longer deterministic
makes the analysis slightly different, but 
$S_n$ is still a random walk in this case.

As will be explained in much more detail below, the case of
site percolation on 
quadrangulations presents further  challenges.  Not only is the
invariant boundary condition not deterministic, it also exhibits
complicated dependencies. In fact, we have not been able to
describe it explicitly.  The process $S_n$ is no longer a
Markov chain but needs to be analyzed as a function of the evolving
boundary.   Our results
regarding the percolation threshold $p_\crit$ for site percolation on
uniform quadrangulations are stated in 
Theorem~\ref{main_thm}.  Apart from the results
themselves, we hope that the methods we present in this work
can be useful for future work on site percolation on other random
maps.  

\begin{remark}\label{richier-rk}
Shortly after making this work public 
there appeared a paper by Lo\"ic 
Richier~\cite{richier} 
in which he (amongst many other things) determined
the exact value of the percolation threshold to
be $p_\crit=\tfrac59$, a stronger result than our bounds
given in Proposition~\ref{bounds_prop} 
and~\eqref{best-bounds}
below.  His method does not involve analyzing an
`evoloving boundary condition' $X_n$ as ours do, and 
is technically simpler.  On the other hand, we also
present some results on invariance of the 
percolation threshold for a variety of boundary
conditions, see Proposition~\ref{amenable_prop}.
We emphasise that these works were
carried out independently.
\end{remark}

\subsection{Problem setting and background}

We now recall the definition of the uniform infinite 
quadrangulation of the half-plane (\uihpq). 
We start with an integer $m\geq1$ 
and a $2m$-gon embedded in the sphere $\SS^2$.  We 
\emph{root} this
polygon by singling out an edge and an orientation of that
edge, and we think of the face on the right of the root
edge as the external face or `outside'.
Let $\phi(n,2m)$ denote the number of quadrangulations
of the inside of the polygon that have $n$ internal 
vertices,  viewed up to 
orientation-preserving homeomorphisms of the sphere.
Thus $\phi(n,2m)$ is a finite number.  
(By convention $\phi(0,2)=1$, counting the `quadrangulation' 
consisting of one single edge only.  Also, $\phi(n,m)=0$ for odd $m$
since quadrangulations are bipartite.)
We define the
uniform distribution $\theta_{n,2m}(\cdot)$
by assigning the same probability $1/\phi(n,2m)$
to each such quadrangulation.  

The {\uihpq} is defined as the weak limit
\be\label{weak-lim_eq}
\theta_{\oo,\oo}=\lim_{m\to\oo}\lim_{n\to\oo}\theta_{n,2m}
\ee
in the \emph{local topology}.  This topology is closely
related to the one introduced
by Benjamini and Schramm~\cite{benjamini-schramm},
and is defined by the following metric
on embedded planar graphs.  For any rooted graph $M$
embedded in $\SS^2$ and any $r\geq1$, let $B_r(M)$
denote the embedded graph spanned by vertices
 at graph distance $\leq r$ from the root edge.
For two such graphs $M_1,M_2$ define
\be
d_{\mathrm{loc}}(M_1,M_2)=
\frac{1}{1+\sup\{r\geq1: B_r(M_1)=B_r(M_2)\}},
\ee
where the equaility $B_r(M_1)=B_r(M_2)$ is interpreted 
in the sense of equivalence under deformations
of the sphere, as before.
The local topology is by definition the topology
generated by the metric $d_{\mathrm{loc}}$.
Existence of the limits~\eqref{weak-lim_eq}
in the local topology
goes back to~\cite{angel:2003,angel:2005}
and~\cite{curien-miermont}.
 We let $\uq$ denote a random variable sampled from the 
{\uihpq} measure $\theta_{\oo,\oo}(\cdot)$ and note here
that $\uq$ is almost surely an infinite rooted graph 
embedded in the plane, with an infinite simple boundary.

In analyzing the limit~\eqref{weak-lim_eq}, the following 
combinatorial facts are of 
central importance~\cite{bg:2009}.  First, the 
asymptotics of $\phi(n,2m)$ are as follows:
\be\label{phi-asy}
\begin{split}
\phi(n,2m)&\sim C_{2m}\rho^n n^{-5/2}, \mbox{ as }n\to\oo,\\
C_{2m}&\sim K \a^{2m} m^{1/2}, \mbox{ as }m\to\oo,
\end{split}\ee
where $\rho=12$, $\a=\sqrt{54}$ and $K$ is a constant.
The generating function $\sum_{n\geq0}\phi(n,2m)z^n$
is thus convergent for $|z|\leq1/\rho$, and its value
at $z=1/\rho$ is known explicitly and denoted
\be\label{Z_eq}
Z(2m)=\sum_{n\geq0}\phi(n,2m)\rho^{-n}=
\frac{8^m(3m-4)!}{(m-2)!(2m)!}.
\ee
(For $m=1$ we interpret $Z(2)$ as its limiting value
$4/3$.)

We now give a rough description of the peeling process,
more details are given in Section~\ref{peeling_sec}.
The boundary of $\uq$ may be identified with $\ZZ$
and each edge on the boundary with a pair 
$(i,i+1)$.  By convention we take the root edge
as $(-1,0)$ pointing towards $0$ (thus the quadrangulation
is in the \emph{upper} half plane).  The peeling process proceeds
by picking an edge on the boundary and `discovering'
the (unique) face $f$ on its left.  This face may have 0, 1
or 2 of its remaining 2 vertices on the boundary,
see Figure~\ref{f:peelperc}.  The probabilities of all the 
different possibilities for $f$ may be computed explicitly
using~\eqref{phi-asy}, and are given in terms of the numbers
\begin{equation}\label{qs_eq}
\begin{split}
&q_{2k}=q_{2k+1}=Z(2k+2)\rho^{-1}\a^{-2k},\quad\mbox{for }k\geq 0,\\
&q_{2k_1+1,2k_2+1}=\Big(\frac{\rho}{\a}\Big)^2q_{2k_1}q_{2k_2},
\quad\mbox{for }k_1,k_2\geq 0,\\
&q_{-1}=\Big(\frac{\a}{\rho}\Big)^2=\frac{3}{8}.
\end{split}
\end{equation}
Roughly speaking, $q_{-1}$ is the probability that $f$ has
no further vertices on the boundary, $q_k$ is the probability
that it has one further vertex on the boundary at distance $k$ 
from the peeling edge, and $q_{k_1,k_2}$ is the probability
that it has two further vertices on the boundary 
at distances $k_1$ and $k_2$,
see~\eqref{p1} and~\eqref{1k}.
The face $f$ divides $\uq$ into two parts, one `above' $f$
and one `below'.   We may redefine
the boundary by `forgetting' the lower part.
  It is a consequence
of the domain Markov property  that 
the remaining, upper part 
also has the law of $\uq$.   One may thus
continue to discover the rest of $\uq$ by repeating the
steps above.  

Similarly to the approach pioneered by 
Angel~\cite{angel:2003}, we will couple the peeling
process with the discovery of the boundary of the
percolation cluster.  Before starting the peeling,
we begin by colouring all 
vertices on the boundary $\ZZ$ black, except for the
vertex $0$ which is coloured white.  Each time we discover
a face we will sample (independently of everything else)
the colours black/white of any vertices on $f$
not discovered
in a previous step.  We describe this procedure fully 
in Section~\ref{perc-peel_sec},
but as an example of what may happen, suppose that at
the very first peeling move we choose the root edge 
$(-1,0)$ to peel from, and that the face $f$ discovered
has two of its vertices in the interior of $\uq$, see 
Fig.~\ref{f:mixedpeel}.
\begin{figure} [t]
\centerline{\scalebox{0.7}{\includegraphics{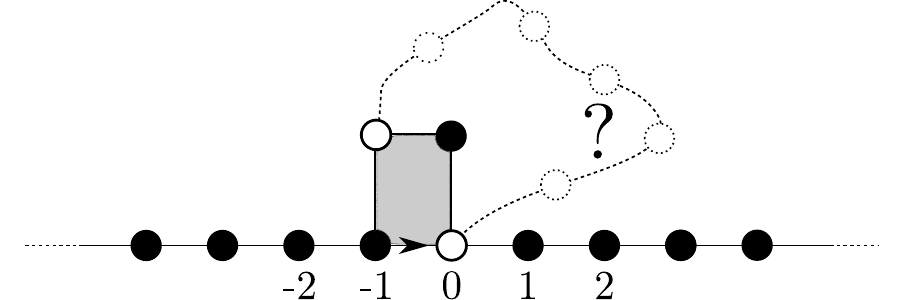}}}
\captionsetup{width=5cm}
\caption{A new white vertex is discovered by exposing the gray
  face. At this stage it is not known whether it belongs to the white
  cluster containing 0. A possible existence of such a connection is
  indicated by dashed lines.} \label{f:mixedpeel}
\end{figure} 
Suppose, furthermore, that these vertices receive
colours black and white, counting counterclockwise 
starting from 0.  At this stage we cannot determine
whether or not the new white vertex
will be part of the white cluster containing 0 or not.
(This differs markedly from the case of triangulations,
where the face $f$ could have at most one new vertex
and this new vertex would have had to be part of the
cluster if it was white.)
We will deal with this by storing, at each step of the
peeling process, information about the
`yet-to-be-decided' white vertices in what we call
a \emph{mixed boundary}.  At step $n$ we denote the 
mixed boundary by $X_n=(X_n(j):j\geq1)$, 
which is a sequence
of colours black and white, indexed by $\NN$.
Understanding the process $(X_n)_{n\geq0}$
is central to our approach.  

\subsection{Main results}

For each sample of $\uq$ and of the 
colours black/white, let $C_0$ denote the 
connected component in the white subgraph containing
the boundary point 0, and let $|C_0|$ denote its
size (number of vertices, say). 
It turns out (as we will show) that there is a number
$p_\crit\in[0,1]$ such that $\PP(|C_0|=\oo)=0$
if $p<p_\crit$ and $\PP(|C_0|=\oo)>0$
if $p>p_\crit$.  Our primary objective has been to pinpoint the
value of the percolation threshold $p_\crit$.
(It is is clear 
that for each $\uq$, the conditional probability
that $|C_0|=\oo$ given $\uq$ is weakly increasing 
in the percolation parameter $p$.
Thus there is a critical value 
$p_\crit(\uq)\in[0,1]$ such that 
this probability is positive if $p>p_\crit(\uq)$
and zero if $p<p_\crit(\uq)$.
Yet the existence of  $p_\crit$ as defined above
requires an argument since there is no obvious
monotonicity when we average also over $\uq$.)

In our attempts to determine the value of 
$p_\crit$, we were led to 
consider percolation on $\uq$ with various boundary
conditions, i.e.\ different ways of assigning black and white colours
to the vertices on the boundary $\ZZ$.  
In our setting, the vertex 0 is always
coloured white, and the vertices $1,2,3,\dotsc$ 
on the boundary to the right of $0$
are always coloured black (although it would 
be straightforward to generalize to other
possibilities for the right side).  The vertices 
$-1,-2,-3,\dotsc$ on the boundary 
to the left of $0$ receive some random
or deterministic colours,  described by 
a vector $\xi\in\S:=\{\b,\w\}^{\{1,2,3,\dotsc\}}$.
Here $\xi(j)\in\{\b,\w\}$ denotes the
colour of $-j$.  We will denote the distribution of $\xi$
by $\nu$.
A-priori, the critical value $p_\crit$ may depend
on the distribution $\nu$, and we write $p_\crit^\nu$
for this critical value.  In case $\xi$ is supported
on the trivial all-black configuration we write 
$\xi\equiv\b$ and $p_\crit^\b$ for the corresponding
percolation threshold.

Next we recall the concept of stochastic ordering.
We order the elements of $\S$ by saying that
$\xi\leq\xi'$ if whenever $\xi(j)=\w$
then also $\xi'(j)=\w$.  We say that an event
(subset) $A\se\S$ is \emph{increasing}
if $\xi\in A$ implies that $\xi'\in A$
whenever $\xi\leq\xi'$.
For two probability measures $\mu,\mu'$
on $\S$ we say that $\mu'$
(stochastically) \emph{dominates} 
$\mu$ if for all increasing
events $A$ we have that $\mu(A)\leq\mu'(A)$.
We will denote the probability measure which assigns
the values $\xi(j)\in\{\b,\w\}$ independently,
with probability $p$ for $\w$, by \iid($p$).

We say that the  random
boundary 
condition $\xi$, or equivalently its distribution $\nu$, 
is \emph{amenable} if there exists 
$p'<1$ such that  $\nu$  is dominated by \iid($p'$).
Note that  $\xi\equiv\b$ is amenable for any $p'<1$,
and that the `free' boundary condition
$\nu=\,$\iid($p$) is amenable whenever $p<1$.
We have the following result on invariance of the
percolation threshold on the boundary condition.

\begin{proposition}\label{amenable_prop}
If $\nu$ is amenable then $p^\nu_\crit=p_\crit^\b$.
\end{proposition}
We will write simply $p_\crit$ for the common critical value.
 Similar results on equality of percolation
thresholds under different boundary conditions have been 
obtained in e.g.~\cite{angel:2013,ray:2013},
usually formulated  
for all-black and {\iid}($p$) boundary conditions.

An important part of our approach is to
analyze the mixed boundary
$(X_n)_{n\geq0}$.  As we will see,  this is
a Markov process, and it has a stationary 
limiting distribution.  We denote a sample from this
limiting distribution by 
$\xi^{(p)} = (\xi^{(p)}(j):j\geq1)$;
we show in Lemma~\ref{iid_lem} that 
$\xi^{(p)}$ is amenable.
Also define for all $k\geq1$
\begin{equation}\label{q-prime}
q'_{2k}=q_{2k} +
\sum_{\substack{k_1+k_2 = 2k\\ k_1,k_2\geq 1,\text{odd}}}
q_{k_1,k_2}.
\end{equation}
In what follows, white vertices are denoted $\w$
and black vertices $\b$.
The following are the results we have obtained.

\begin {theorem}\label{main_thm}
The percolation threshold $p_\crit$ 
satisfies
\begin{equation}
p_\crit\left\{
\begin{array}{ll}
\geq \sup\{p\in[0,1]:\a(p)<0\}\\
\leq \inf\{p\in[0,1]:\a(p)>0\},
\end{array}
\right.
\end{equation}
where
\begin {equation}\label{alpha-eq}
 \alpha(p) = \frac{3}{8} p^2 + \frac{5}{8} p - \frac{1}{2} + \beta(p)
\end {equation}
and 
\begin{equation}
\beta(p)=\Big(p+\frac{1}{3}\Big)
\sum_{k\geq1,\mathrm{odd}}\PP(\xi^{(p)}(k+1)=\w)q_k+
\sum_{k\geq1,\mathrm{even}}\PP(\xi^{(p)}(k+1)=\w)q'_k.
\end{equation}
\end {theorem}

The function $\a(p)$ depends on the distribution
of $\xi^{(p)}$, which
we have not been able to find explicitly. 
However we do
derive  upper and lower bounds on 
the probabilities $\PP(\xi^{(p)}(k+1)=\w)$, resulting in 
bounds on $p_\crit$.  In particular, we obtain
the following:

\begin {proposition}\label{bounds_prop}
 It holds that 
$0.511 \approx \frac{\sqrt{493}-13}{18} \leq p_\crit 
\leq\frac{\sqrt{73}-5}{6}\approx 0.591$.
\end {proposition}

In fact, we present a method for deriving
better bounds for $p_\crit$, which when implemented
numerically on a computer gives the bounds
\be\label{best-bounds}
0.5511 \leq p_\crit\leq 0.5581
\ee
mentioned in the Abstract.
(The bounds presented in
Proposition~\ref{bounds_prop} were chosen as they are simple algebraic
expressions obtained without recourse to numerical methods.)
Our basic approach is to find
probability measures which are stochastically above and below the law
of $\xi^{(p)}$, thus giving upper and lower bounds on $\a(p)$.  
These measures are obtained as the 
stationary distributions of  certain finite state space Markov chains,
meaning that they can be found by solving deterministic 
equations involving $p$.  We expect that the gap
in~\eqref{best-bounds} could be narrowed by further increasing the
size of the state space.
Note that we do not necessarily expect  
 $p_\crit$ to be given by a simple formula, 
indeed our bounds quickly become too complicated
to write down by hand, which is why we have used
numerical methods.


\subsection{Outline}

In Section~\ref{peeling_sec} we properly define the peeling
process, and also provide some basic results about
the critical probability $p_\crit$.
In Section~\ref{mixed_sec} we study the process
$(X_n)_{n\geq0}$ in detail.  We then apply our results
on $(X_n)_{n\geq0}$ in Section~\ref{pc_sec}, where
we prove Theorem~\ref{main_thm} and 
Proposition~\ref{bounds_prop}. 
Furthermore, we explain 
how one may obtain increasingly better upper 
and lower bounds on $p_c$. 

We note here that often the dependency 
on $p$ will be dropped in the notation 
as e.g.~in $X_n$ and $S_n$.

\section{Peeling process}\label{peeling_sec}

The peeling process gives a sequence
$\uq=\uq_0\supset \uq_1\supset\cdots$  of random
infinite quadrangulations with  infinite simple boundary.
At each step $n\geq0$ there is a choice  of an edge $r_n$
on the boundary of $\uq_n$, which we require to
be independent of $\uq_n$ itself,
and given this edge we obtain $\uq_{n+1}$
by a random operation $\peel$ such that
\begin {equation} \uq_{n+1} = \peel(\uq_{n},r_n).
\end {equation} 
The operation $\peel$ is defined as follows.
We start by discovering the unique face $f_n$
of $\uq_n$ adjacent to $r_n$ (the distribution
of $f_n$ will be given shortly).  This face may
be adjacent to  0, 1 or 2 internal vertices,  
the remaining vertices 
are on the boundary of $\uq_n$ and may be to the left
or to the right of $r_n$, as illustrated in 
Figure~\ref{f:peelperc}.  If $f_n$ has vertices on the
boundary then it encloses one or two subquandrangulations
of $\uq_n$, each with a finite, simple boundary. 
By definition, $\uq_{n+1}$ is the infinite
quadrangulation obtained by  removing $f_n$ along with any
such enclosed subquadrangulations.  The edges and vertices
of $f_n$ that then become part of the boundary
of $\uq_{n+1}$ are called \emph{exposed} edges and vertices,
respectively.  The number of exposed edges 
is denoted by $\Ex_n$.
The  edges on the boundary
of $\uq_n$ which are enclosed by $f_n$ are
called \emph{swallowed edges}, and the numbers
of such edges  to the left and right of $r_n$
are denoted by $\Sw^-_n$ and $\Sw^+_n$,
respectively.

Recall the $q$:s  in~\eqref{qs_eq}
and~\eqref{q-prime}.  We note for future reference 
that they satisfy
\begin{equation}\label{q-sums}
\begin{split}
&\sum_{k\geq0}q_{2k+1}=\sum_{k\geq0}q_{2k}=\tfrac18,\quad
\sum_{k_1,k_2\geq0}q_{2k_1+1,2k_2+1}=\tfrac{1}{24}\\
&\sum_{k\geq1}q'_{2k}=\tfrac18-\tfrac19+\tfrac{1}{24}=\tfrac{1}{18}.
\end{split}
\end{equation}
As long as the peeling edge $r_n$ is always
chosen independently of $\uq_n$
we have the following~\cite{angel:2013}:
\begin{list}{$\bullet$}{\leftmargin=1em}
 \item $\uq_n$ is distributed as $\uq$ for all $n\geq 0$ and is
independent of the previous steps.
 \item The couples $(\Ex_n,\Sw_n^-,\Sw_n^+)$ form an
i.i.d.\ sequence, each being  independent of $\uq_n$.
\item If the revealed
face $f_n$ contains vertices in the boundary 
on both sides of $r_n$ then 
$\Sw_n^+ > 0$ and $\Sw_n^- > 0$ and then
necessarily $\Ex_n = 1$.
We have
\begin {equation} \label{p1}
\P(\Ex_n=1,\Sw_n^+=k_1,\Sw_n^-=k_2)=q_{k_1,k_2}
\end {equation} for $k_1,k_2 \geq 1$ both odd.
\item  Otherwise either 
$\Sw_n^- = 0$ or $\Sw_n^+ = 0$  and then we have
\begin {equation} 
\P(\Ex_n=e,\Sw_n^\pm =k) =
\left\{\begin{array}{ll} 
q_{-1}=\frac{3}{8}, & e=3, k=0, \\
q_{k}\indic{k~\text{odd}}, & e=2, k\geq 1, \\
q'_k\indic{k~\text{even}} 
+ \frac{1}{3}q_k\indic{k~\text{odd}}, & 
e =1, k\geq 1. \label{1k}
                             \end{array}\right.
 \end {equation}
It follows from this and~\eqref{q-sums} that
\begin{equation}
\P(\Ex_n=2,\Sw_n^\pm = 0) = \tfrac{1}{8}, \quad
\P(\Ex_n=1,\Sw_n^\pm = 0) =\tfrac{5}{18}. \\ 
 \end {equation}
 \item The expectations
$\E(\Ex_n) = 2$ and $\E(\Sw_n^\pm) = 1/2$.
\end{list}

The enclosed subquadrangulations which are removed to form 
$\uq_{n+1}$ from $\uq_n$ are almost surely finite, and independent
both of each other (if there are two) and of $\uq_{n+1}$.
If the perimiter of an enclosed quadrangulation is $2m$,
then~\eqref{Z_eq} gives the partition function of its 
distribution, but in this paper we shall only use the fact that
it is almost surely finite.

\subsection{Percolation and peeling}
\label{perc-peel_sec}

Recall that the boundary of $\uq$ is identified with $\ZZ$, and that
the directed edge from $-1$ to 0 is taken as root.  
Also recall our class of boundary
conditions:  The vertex 0 is always
coloured white, the vertices $1,2,3,\dotsc$ 
on the boundary to the right of $0$
are always coloured black, and the vertices 
$-1,-2,-3,\dotsc$ on the boundary 
to the left of $0$ receive some random
or deterministic colours,  described by 
a vector $\xi\in\S$ with law $\nu$.

For reasons that will appear later we mainly consider 
$\xi$
which are \emph{admissible}, which we define to mean that
(i) $\xi(1)=\b$, and (ii) for all $k\geq1$, 
if $\xi(k)=\w$ then
$\xi(k+1)=\b$ (i.e., there are no adjacent white vertices).
 We let $\hat\S\se\S$ denote the set of all
admissible boundary conditions.
The canonical admissible boundary
condition is obtained when all vertices to the left of 0 are black, 
written 
$\xi\equiv\b$. 

The remaining vertices, i.e.\ those not on the boundary, 
are coloured independently of each other and of $\xi$,
each being white with
probability $p$ and black with probability $1-p$.

We are interested in knowing for which values of $p$ the 
white cluster $C_0$ percolates,
i.e.\ is infinite with positive probability.  We will investigate
this using the peeling process where we discover the colours at the
same time as we peel.  It remains to define the peeling edges $r_n$.
In the first step we let $r_0=(-1,0)$. 
Assuming that we have
defined $r_{n}=(r_n^L,r_n^R)$ for some $n \geq 0$, 
we reveal the new face $f_n$ and then reveal the colours of 
all exposed vertices on $f_n$. 
There will be a few different 
cases for the next peeling edge $r_{n+1}$. 
The first case we consider is when $\Sw^+_n=0$, that is to say that 
no edges on the boundary to the right of $r_{n}$ are swallowed.
 Starting from the vertex immediately to the left of
$r_n^R$ in the boundary of $\uq_{n+1}$, follow this boundary from
right to left until the first black vertex is encountered.  This black
vertex is denoted $r_{n+1}^L$, the vertex immediately to its right
in the boundary of $\uq_{n+1}$ is denoted $r_{n+1}^R$, and
$r_{n+1}=(r_{n+1}^L,r_{n+1}^R)$ is taken as the new peeling edge.
If, on the other hand, $\Sw_n^+>0$, then 
we denote by $s_n$ the rightmost vertex on the
boundary of $\uq_n$ belonging to the new face $f_n$. 
We then follow the boundary of $\uq_{n+1}$ from right to left starting
at the vertex immediately to the left of
$s_n$, until we discover a black vertex 
which we take to be $r_{n+1}^L$.  The vertex immediately to its
right is taken to be $r_{n+1}^R$.
Examples are given in Figure \ref{f:peelperc}.
Note that the choice of $r_{n+1}$ is always independent of
$\uq_{n+1}$.
\begin{figure} [t]
\centerline{\scalebox{0.45 }{\includegraphics{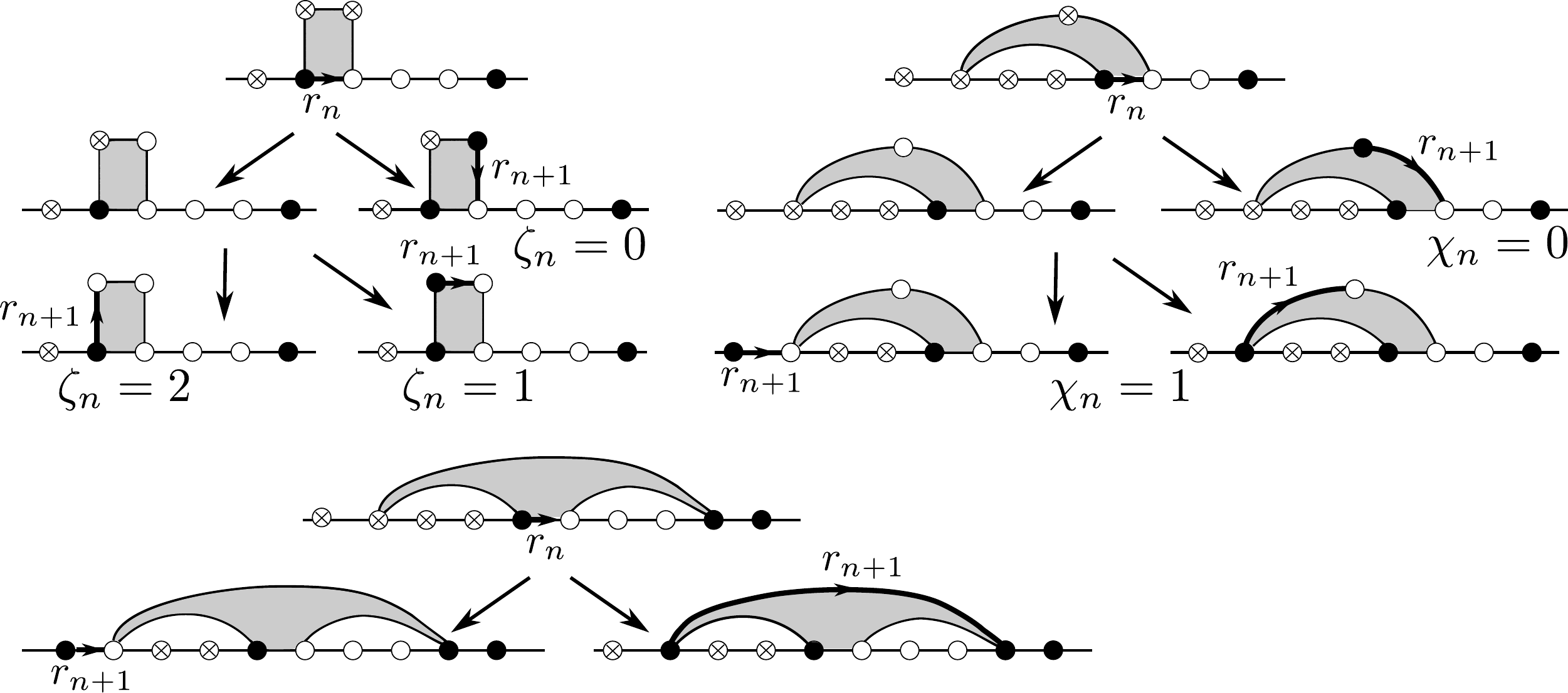}}}
\caption{Examples of the definition of the peeling edge $r_n$,
as well as the
random variables $\zeta_n$ and $\chi_n$ 
(defined in Section~\ref{pf_sec}). 
Unspecified colours are
represented by $\otimes$.} \label{f:peelperc}
\end{figure} 

In each step, $\uq_n$ will have a
\emph{mixed--white--black} boundary condition.
The \emph{white}  part of the boundary is
possibly empty;  if non-empty it leftmost endpoint is marked by 
$r_{n}^R$.   We will denote
the number of white vertices in this white arc 
\emph{which belong to} $C_0$ 
by $S_n$.   Thus $S_0=1$ and if $S_n=0$
then $S_{n'}=0$ for all $n'\geq n$.
The mixed part of the boundary is 
described by a vector $X_n=(X_n(j):j\geq1)\in\S$,
where $X_n(j)$ denotes
the colour of the $j$:th vertex to the left of $r_{n}^R$.
In fact it is easy to see
that if $\xi=X_0$ is admissible then $X_n$ is admissible for 
all $n\geq0$.

Although we have defined this procedure for all $n\geq0$, we are
primarly concerned with the process up to the first time that $S_n=0$
(if this ever happens).  Until this time, all vertices in the \emph{white}
arc contribute to $S_n$, and the state of the mixed boundary in $\uq_n$ is
thus determined by $(X_n, S_n)\in\hat\S\times\NN$.  This
is easily seen to be a Markov chain. 
Note that the process $(X_n)_{n\geq 0}$ is  itself a Markov chain,
but the process $(S_n)_{n\geq0}$ is not;  whereas $X_{n+1}$ only
depends on $X_n$ as well as the next peeling move, to determine
$S_{n+1}$ we must also look at $X_n$ since some of the white vertices
in $X_n$ may become part of the white cluster.

The key observation, which this model shares with other studies of
percolation using the peeling process, is that
$C_0$ must be finite  
if ever $S_n=0$. 
 Note that this holds even if the initial
boundary condition $\xi$ contains infinitely many white vertices,
since only finitely many of them can become connected to 0 during a
finite number of peeling steps.  We omit formal proof of this but
state it as a proposition:
\begin{proposition}\label{Sn-zero_prop}
If $S_n=0$ for some $n$ then $C_0$ is finite.
\end{proposition}

\subsection{Percolation threshold for different boundary conditions}


In this section we show that there is indeed a percolation
threshold $p_\crit^\b$ when the boundary condition
$\xi\equiv\b$ is all-black (as stated in Section~\ref{intro_sec}).
We will also prove Proposition~\ref{amenable_prop}.

The following result is not new but we include a 
proof for completeness.

\begin {proposition}\label{pc_prop}
For each $\nu$ supported on $\hat\S$ we have that
\be\label{pc1}
\mbox{if $p$ is such that } \PP(S_n\to0)=1\mbox{ then }
\PP(|C_0|=\oo)=0 \mbox{ for all }p'\leq p,
\ee
and 
\be\label{pc2}
\mbox{if $p$ is such that } \PP(S_n\to0)<1\mbox{ then }
\PP(|C_0|=\oo)>0 \mbox{ for all }p'\geq p.
\ee
In particular, there is a value
$p_\crit^\nu\in[0,1]$ such that
$\PP(|C_0|=\oo)=0$ if $p<p_\crit^\nu$, and
$\PP(|C_0|=\oo)>0$ if $p>p_\crit^\nu$.
\end {proposition}

Here $p_\crit^\nu$ may be identified either as the supremum 
of $p$'s  as in~\eqref{pc1}, or the infimum of
$p$'s  as in~\eqref{pc2}.

\begin {proof}
Write $P^\uq_p$ for the conditional measure $\PP(\cdot\mid\uq)$,
thus $P^\uq_p$ governs the black/white colours only.
 An obvious coupling gives, for each $\uq$, that
\be\label{mon_pc}
\mbox{if } p'\leq p\mbox{ then }
P^\uq_{p'}(|C_0|=\oo)\leq P^\uq_p(|C_0|=\oo).
\ee

To see~\eqref{pc1}, suppose $p$ is such that $\P(S_n\to 0) = 1$.
By Proposition~\ref{Sn-zero_prop} we thus have that
\be
0=\PP(|C_0|=\oo)=\EE[P^\uq_p(|C_0|=\oo)].
\ee
Hence $P^\uq_p(|C_0|=\oo)=0$ almost surely, and so 
by~\eqref{mon_pc} also $P^\uq_{p'}(|C_0|=\oo)=0$ almost surely,
which gives the result.

For~\eqref{pc2}, first let $A$ be the event that, during the
peeling process, there are infinitely many times when the revealed
face $f_n$ has two exposed vertices and they are both white.  Since
the sequence of revealed faces and the colours of the exposed vertices
is i.i.d.\ it follows that $\PP(A)=1$.  It follows that for $p$ as
in~\eqref{pc2} we have
\be
\PP(|C_0|=\oo)\geq \PP(\{S_n>0\;\forall n\geq0\}\cap A)>0.
\ee
This means that $\EE[P^\uq_p(|C_0|=\oo)]>0$ and hence 
(using~\eqref{mon_pc}) for all $p'\geq p$
also $\EE[P^\uq_{p'}(|C_0|=\oo)]>0$, as required.
\end {proof}


Recall that $\xi$, or equivalently its distribution $\nu$, 
is amenable if there exists 
$p'<1$ such that  $\nu$  is dominated by \iid($p'$),
and that 
Proposition~\ref{amenable_prop} claims that
$p^\nu_\crit=p_\crit^\b$ whenever $\nu$ is amenable.

\begin{proof}[Proof of Proposition~\ref{amenable_prop}] 
In this proof we let $\PP$ denote the probability measure under 
which $\xi$ has law $\nu$.  We need to show that, firstly,
if $p>p_\crit^\b$ then $\PP(|C_0|=\oo)>0$, and secondly,
if $p<p_\crit^\b$ then $\PP(|C_0|=\oo)=0$.

The first statement is clear:  for each joint realization of $\uq$ and
the black/white colours such that $|C_0|=\oo$ when the boundary
is all-black,
then also $|C_0|=\oo$ for the boundary condition $\xi$.
We now turn to the second statement.

Let $\uq^\boxtimes$ denote the
\emph{matching lattice} of $\uq$,
that is the graph obtained from $\uq$ by adding to the
edge set both diagonals of each face (this graph is in general not planar).
Then $C_0$  is finite if and only if there are vertices
$-i<0$ and $j>0$ on the boundary $\ZZ$ such that both are black 
(equivalently, 
$\xi(i)=\b$) and there is a path from $-i$ to $j$ in $\uq^\boxtimes$
traversing only black vertices,
see e.g.~\cite[Section~2.2]{kesten} 
or~\cite[Section~11.10]{grimmett}.
We call such a path a 
\emph{blocking  circuit}.  
We need to show that with probability 1 
there is  a blocking
circuit when the boundary condition is $\xi$.

For each $j\in\ZZ$, let $B_j$ denote the event that  there are $i<j$
and $k>j$ such that there is a path in $\uq^\boxtimes$ from $i$ to $k$
which (i) contains no other vertex of $\ZZ$, and (ii) contains only
black vertices, apart from possibly $i$ and $k$.  Since $p<p_\crit^\b$
we have that $\PP(B_0)=1$. By invariance under translation
of $\uq$ with respect to the boundary we thus have that
$\PP(B_j)=1$ for all $j\in\ZZ$.  
Consider the event
$B=\cap_{j<0}B_j$.  We have that $\PP(B)=1$, and on the event $B$
there are infinitely many $i>0$ such
that there is a $\uq^\boxtimes$-path from $-i$ to some $j>0$
which contains no other vertex of $\ZZ$ and consists only of black
vertices, apart from possibly $-i$.  
(To see this, note
that some of the paths whose existence are guaranteed by the $B_j$ may
`merge'.)  Since $\xi$ is amenable, at least one (in fact infinitely
many) of these vertices $-i$ are black.
Hence there is a blocking circuit with probability one.
\end{proof}

The argument for
Proposition~\ref{amenable_prop} applies more generally, e.g.\
it is enough if for any fixed sequence $i_1<i_2<i_3<\dotsb$,
with $\nu$-probability 1 at least one $\xi(i_k)=\b$.
One can also adapt the argument to cases when one
has a general amenable boundary condition also
to the right of 0.

\section{Evolution of the mixed boundary}
\label{mixed_sec}

Let $\S^\ast=\bigcup_{k\geq1}\{\b,\w\}^k$ denote the set of all \emph{finite} 
sequences of $\b$'s and $\w$'s
of length at least 1, and  let 
$\hat\S^\ast$ denote the subset of $\S^\ast$
consisting of all sequences such that the
first bit  is $\b$ and such that 
there are no adjacent $\w$'s.  
Clearly $\S^\ast$ and
$\hat\S^\ast$ are countable.
We endow $\S$ with the
product topology and $\hat\S$ with the subspace topology.
Note that we may define a metric $d_{\S}$ on 
$\S$ which generates its topology by e.g.
\begin {equation}
 d_{\S}(x,y) = \left(\sup\{k~:~[x]_k = [y]_k\}+1\right)^{-1}
\end {equation}
where $[x]_k\in\S^\ast$ denotes the vector consisting of the first $k$
entries in $x$.
Since $\hat\S$ is closed in $\S$ and $\S$ is compact, 
it follows that also  $\hat\S$ is compact.  

We now turn to investigating existence and uniqueness
of invariant distributions for the 
process $(X_n)_{n\geq0}$.  We begin by noting the following 
immediate consequence of the Stone--Weierstrass theorem:

\begin {lemma}\label{l:dense}
 Let $C(\hat\S,\mathbb{R})$ be the set of continuous real valued functions
 on $\hat\S$ equipped with the uniform topology. 
Let $C^\ast(\hat\S,\mathbb{R})$ be
 the subalgebra of functions which only depend on finitely many
 coordinates. Then $C^\ast(\hat\S,\mathbb{R})$ is dense in 
$C(\hat\S,\mathbb{R})$.
\end {lemma}

Next, recall that a Markov process $(Y_n)_{n\geq0}$ on 
$\hat\S$ is \emph{Feller} if
for any $n\geq 0$ and any bounded continuous
 function $g: \hat\S\rightarrow \R$, the 
function $\xi\mapsto\E_\xi(g(Y_n))$ is
 continuous in $\xi$ (where $\E_\xi$ is the 
expected value given that
 $Y_0=\xi$).

\begin {lemma}
 The process  $(X_n)_{n\geq 0}$ is Feller.
\end {lemma}
\begin {proof}
Fix $\xi\in \hat\S$, 
 $n\geq0$ and a bounded continuous $g:\hat\S\to\RR$.
 Let $\epsilon > 0$.  We will show that there exists
$\delta > 0$ such that
 \begin {equation}
  |\E_{\xi}(g(X_n))-\E_{\xi'}(g(X_n))| < \epsilon
 \end {equation}
for every $\xi'\in\hat\S$ such that 
$d_{\S}(\xi,\xi') < \delta$.  Let
\begin {equation}
A_n(k) = \left\{\sum_{i=0}^n \Sw_i^- \leq k\right\}
\end {equation}
i.e.~$A_n(k)$ is the event that we swallow no more than $k$ edges 
on the left in
the first $n+1$ steps.  Using that $g$ is bounded, let $C>0$ be a
constant such that $\sup g < C$. Since the $\Sw_i^-$ are a.s.~finite
we may choose $k$ large enough such that $\P(A_n^c(k)) <
\epsilon/(4C)$. Since $g$ is continuous,  by Lemma~\ref{l:dense}
one may choose $j$ large enough such that there is a function 
$g_j \in C^\ast(\S,\mathbb{R})$ 
which depends only on the first $j$ coordinates
obeying $\sup|g-g_j| < \epsilon/4$. Then
\begin {align}
   |\E_{\xi}(g(X_n))-\E_{\xi'}(g(X_n))| &\leq
   |\E_{\xi}(g(X_n))\indic{A_n(k)})-\E_{\xi'}(g(X_n)\indic{A_n(k)})|
   \nonumber \\ &\quad+
   |\E_{\xi}(g(X_n)\indic{A^c_n(k)})-\E_{\xi'}(g(X_n)\indic{A^c_n(k)})|
   \nonumber \\ &\leq
   |\E_{\xi}(g_j(X_n)\indic{A_n(k)})-\E_{\xi'}(g_j(X_n)\indic{A_n(k)})|
   \nonumber \\ &\quad+ 2\sup|g-g_j|+ 2C \P(A^c_n(k)) \nonumber \\ &<
   |\E_{\xi}(g_j(X_n)\indic{A_n(k)})-\E_{\xi'}(g_j(X_n)\indic{A_n(k)})|
   \nonumber \\ &\quad+ \epsilon.
\end {align}
Now choose $\delta < 1/(1+j+k)$. Then, if 
$d_{\S}(\xi,\xi') < \delta$ it
holds that $[\xi]_{j+k}=[\xi']_{j+k}$. Since $g_j$ depends only on the
first $j$ coordinates it thus holds that
$\E_{\xi}(g_j(X_n)\indic{A_n(k)})=\E_{\xi'}(g_j(X_n)\indic{A_n(k)})$
and thus
\begin {align}
   |\E_{\xi}(g(X_n))-\E_{\xi'}(g(X_n))| < \epsilon.
\end {align}
\end {proof}
\begin {proposition}\label{kryl_bog}
There is a probability  measure $\mu$ on $\hat\S$
which is invariant for the process $(X_n)_{n\geq0}$.
\end {proposition}
\begin {proof}
This follows from the Krylov--Bogolyubov 
Theorem~\cite[Corollary~4.18]{hairer}
since $\hat\S$ is compact and
$(X_n)_{n\geq 0}$ is Feller.
\end {proof}

Do we expect the invariant distribution $\mu$ to 
be unique?  Imagine placing at time 0
a `flag' on the edge immediately to the left of
$-1$.   At
time $n$ the flag will have moved away from the peeling edge or been
swallowed.  Each time it is swallowed we reset it immediately to 
the left of $-1$.  If the position of the flag
is `recurrent enough' then  $X_n$ will always retain 
information about the initial condition $\xi$, and hence in this
case the distribution $\mu$ cannot be unique.
We make this intuitive sketch more precise now.

As mentioned, at time $n=0$ we mark the edge $(-2,-1)$
just left of $r_0$ by a `flag'.  During the peeling, the 
relative position of the flagged edge with respect to the 
peeling edge $r_n$ will change:  the distance increases
when we `input' into $X_n$, and decreases when we swallow
to the left.  Whenever we swallow the flagged edge we reset
the flag on the edge just to the left of the peeling edge
$r_n$.  We denote the distance between $r_n$ and the
flagged edge by $W_n$.  Thus $W_0=1$ and $W_n\geq1$
for all $n\geq0$.  Let $X^\ast_n\in\hat\S^\ast$ denote the
vector of colours black/white of the vertices between $r_n$ 
and the flagged edge.  So the length $|X^\ast_n|$ of 
$X^\ast_n$ is precisely $W_n$.

Note that the process $(X_n^\ast)_{n\geq0}$ does not depend on the
initial state $\xi\in\hat\S$ of $(X_n)_{n\geq0}$. 
In fact,  the evolution of
$X_n^\ast$ does not depend on $X_n(k)$ for any $k>W_n$.
Thus $(X_n^\ast)_{n\geq0}$ is a Markov 
chain on the countable state space $\hat\S^\ast$,
and it is not hard to see that it is irreducible and aperiodic.
Hence it is either transient, null-recurrent, or 
positive-recurrent, depending on $p$.  
By a slight abuse of terminology, we will refer 
to these three cases
as the \emph{transient}, \emph{null-recurrent}, and 
\emph{positive-recurrent} cases, respectively,
also when referring to the chain 
$(X_n)_{n\geq0}$ itself.

For $\xi^\ast\in\hat\S^\ast$, with length $\ell=|\xi^\ast|$,
and $\xi\in\hat\S$, define the \emph{concatenation} of 
$\xi^\ast$ and $\xi$ as the element $\bar\xi$ of $\S$
given by
\[
\bar\xi(j)=\left\{\begin{array}{ll}
\xi^\ast(j), & \mbox{if } j\leq\ell,\\
\xi(j-\ell), & \mbox{if } j>\ell.
\end{array}\right.
\]
Consider now the positive-recurrent case.
Then standard Markov
chain theory~\cite{norris}
implies that $(X_n^\ast)_{n\geq0}$  has a 
unique invariant distribution $\mu^\ast$ supported
on $\hat\S^\ast$.  
Let $\xi^\ast\in\hat\S^\ast$ denote a random variable
with distribution $\mu^\ast$.
One may
obtain many distributions $\mu$ which are invariant
for the `whole' process $(X_n)_{n\geq0}$ by 
concatenating  $\xi^\ast$ with some
$\xi\in\hat\S$.  
Let $\mu^{(p)}$ be the probability measure on $\hat\S$
which is given as the distribution of $\xi^\ast$
concatenated with $\xi\equiv\b$.  Then $\mu^{(p)}$
 is clearly an invariant measure for 
$(X_n)_{n\geq0}$.
Equivalently, $\mu^{(p)}$ is the limiting distribution of the
process $(X_n)_{n\geq0}$ starting with $X_0\equiv\b$.

We now turn to the transient and null-recurrent
cases.  For any probability measure $\nu$ on
$\hat\S$, let $[\nu]_k$ denote the law of
$[\xi]_k\in\{\b,\w\}^k$ when $\xi$ has distribution $\nu$.  We define
the total-variation distance between two probability measures
$\mu$ and $\nu$ on $\hat\S$ by 
\be
\|\mu-\nu\|_{\mathrm{TV}}=\sup_A |\mu(A)-\nu(A)|
\ee
where the supremum is taken over all measurable sets $A$.  
\begin{proposition}\label{mp-unique_prop}
Let $\xi\in\hat\S$ be random or deterministic and
let $(X_n)_{n\geq0}$ denote the chain started in 
$\xi$.  Let $\mu$ be an invariant distribution
as indentified in Proposition~\ref{kryl_bog}.
In the transient and null-recurrent cases,
we have for each $k\geq1$ that
\begin{equation}\label{tv_eq}
\|[X_n]_k-[\mu]_k\|_{\mathrm{TV}}\to0\quad
\mbox{as }n\to\oo.
\end{equation}
In particular, $\mu$ is unique and may be obtained
as the limiting distribution starting in any state.
\end{proposition}
In these cases we denote the unique invariant 
distribution by $\mu^{(p)}$.  Thus $\mu^{(p)}$ is
uniquely defined for all $p\in[0,1]$.
\begin{proof}
We may couple the chain $X_n^{(\mu)}$ started in $\mu$
with the chain $X_n^{(\xi)}$ started in $\xi$ in the natural way, by
using the same peeling moves.
As noted above, the flag distances are the same in both
processes, we denote this common value by $W_n$.
Also, for all $n$ we have that $X_n^{(\mu)}(k)=X_n^{(\xi)}(k)$
for all $k\leq W_n$.  It follows from the coupling inequality
that
\begin{equation}
\|[X^{(\xi)}_n]_k-[\mu]_k\|_{\mathrm{TV}}
\leq \PP([X^{(\xi)}_n]_k\neq [X_n^{(\mu)}]_k)
\leq \PP(W_n<k).
\end{equation}
The event $\{W_n<k\}$ is precisely the same as the event that
$X^\ast_n$ belongs to the finite subset 
$\{\xi\in\hat\S^\ast:|\xi|<k\}$ of $\hat\S^\ast$.
By standard results for Markov chains, 
e.g.~\cite[Theorem~1.8.5]{norris},
it follows that $\PP(W_n<k)\to0$
as $n\to\oo$ in the transient and null-recurrent cases, 
proving~\eqref{tv_eq}.

Using Prohorov's theorem~\cite[Section~I.5]{billingsley} 
(or Lemma~\ref{l:dense} straight away)
we deduce from~\eqref{tv_eq} that $X^{(\xi)}_n$ 
converges weakly to $\mu$ for
all choices of $\xi$, thereby proving also
the final part of the statement.
\end{proof}

By Proposition~\ref{mp-unique_prop} (and by
definition, in the positive-recurrent case)
for all $p$ the measure $\mu^{(p)}$
can be obtained as the limiting distribution when starting
with initial condition $\xi\equiv\b$.
We let $\xi^{(p)}\in\hat\S$ denote
a random variable with distribution $\mu^{(p)}$.
The next result implies that the
stationary boundary condition $\xi^{(p)}$
is amenable whenever $p<1$:
\begin{lemma}\label{iid_lem}
For all $p$, the measure
$\mu^{(p)}$ is dominated by \iid($p$).
\end{lemma}
\begin{proof}
It suffices to show that 
if $X_0\equiv\b$ then the law of $X_n$ is dominated
by \iid($p$) for all $n\geq0$.  More precisely,
we will show that one may define a process
$(Y_n)_{n\geq0}$ such that for all $n$,
(a) the distribution of $Y_n$ is \iid($p$) and
(b) $X_n\leq Y_n$.  
We show this by induction.

For $n=0$ this clearly
holds if we just sample $Y_0$ from \iid($p$).
Assume that we have such a coupling for the first
$n$ steps in the peeling process.  There are three main
cases to consider depending on the next peeling move.
In the first case, $X_{n+1}=X_n$ and we may take 
also $Y_{n+1}=Y_n$ (this happens e.g.\ if we swallow only to the
right and expose either no vertex or one white vertex).   
The second case is that  we swallow to
the left, and/or input one black vertex.
If we perform the corresponding truncation on $Y_n$,
and if necessary input independently a new bit
(white or black with probability $p$ or $1-p$),
this preserves properties (a) and (b).  The third possibility
is that we reveal 3 edges and thus input to $X_n$
either $\b\w$ or $\b\b$ (read from right to left), 
with relative probabilites 
$p$ and $1-p$.  Again, this may straightforwardly be
coupled with an input of two independent bits
into $Y_n$ so that property (b) is preserved.  This proves the result.
\end{proof}

In the next result we let 
$(V_n)_{n\geq0}$ denote an arbitrary 
sequence of i.i.d.\ random variables in
$\ZZ^k$ such that $V_n$ is independent of $X_n$ for all $n$,
and we let $V$ have the same distribution as the $V_n$ 
and be independent of $\xi^{(p)}$ and of $\xi^\ast$.
\begin{lemma}\label{ergodic_lem}
Let $F:\Sigma\times\ZZ^k\to\RR$ and
$F^\ast:\Sigma^\ast\times\ZZ^k\to\RR$ be bounded and
continuous functions.  
Consider the processes $(X_n)_{n\geq0}$
and $(X^\ast)_{n\geq0}$ started in the
invariant distributions $\mu^{(p)}$ and $\mu^\ast$,
respectively.
\begin{enumerate}
\item In the transient and null-recurrent cases,
\[
\frac{1}{n}\sum_{j=0}^{n-1} F(X_j,V_j)\to
\E[F(\xi^{(p)},V)]\quad \mbox{almost surely}.
\]
\item In the positive-recurrent case,
\[
\frac{1}{n}\sum_{j=0}^{n-1} F^\ast(X_j^\ast,V_j)\to
\E[F^\ast(\xi^\ast,V)]\quad \mbox{almost surely}.
\]
\end{enumerate}
\end{lemma}
\begin{proof}
In either case, the process
$((X_n,V_n))_{n\geq0}$ or 
$((X^\ast_n,V_n))_{n\geq0}$ is a Markov process started
in its  unique invariant distribution.
Hence
 the result follows from a standard 
ergodic theorem for Markov processes, see 
e.g.~\cite[Corollary~5.12]{hairer}.
\end{proof}

\section{The critical probability}
\label{pc_sec}

Consider the peeling process started with the stationary version
$\xi^{(p)}$ of the boundary.  By Lemma~\ref{iid_lem}, this boundary
condition is amenable, and hence by Proposition~\ref{amenable_prop}
the critical probability is equal to $p_\crit=p_\crit^\b$;  that
is, the critical probability is the same as if we had started from an
all-black boundary.  We will now use this together with 
Proposition~\ref{pc_prop} and Lemma~\ref{ergodic_lem}
to relate $p_\crit$ to the function $\a(p)$ 
in Theorem~\ref{main_thm}.

\subsection{Proof of Theorem~\ref{main_thm}}
\label{pf_sec}

 Let $(\zeta_n)_{n\geq 1}$ 
denote a sequence of i.i.d.~random variables
(independent of everything else) satisfying 
\be
\P(\zeta_n=0) = 1-p,\;
\P(\zeta_n = 1) = p(1-p),
\mbox{ and } \P(\zeta_n = 2)=p^2. 
\ee
We identify $\zeta_n$ with the number of consecutive 
new white vertices from right
 to left (starting from the rightmost) on the revealed face $f_n$
 when $\Ex_n = 3$ (see Fig.~\ref{f:peelperc}). 
Let $(\chi_n)_{n\geq1}$ be a sequence of
 i.i.d.~random variables (independent of everything else)
satisfying $\P(\chi_n = 0) = 1-p$ 
and $\P(\chi_n = 1)=p$. 
We identify $\chi_n$ with
the number of new white vertices
on the revealed face $f_n$ when $\Ex_n=2$. 

Let $(\hat S_n)_{n\geq0}\in\ZZ$
be the process given by $\hat S_0=1$ and
 \begin {align}\nonumber
  \hat S_{n+1} = &\hat S_{n} + \indic{\Ex_n = 3} \zeta_n + 
\indic{\Ex_n =2}\Big(\chi_n+
  \chi_n\sum_{k=1}^\infty\indic{X_{n}(k+1) = \circ;\Sw_n^- = k}
  -\Sw_n^+\Big) \\ 
&+\indic{\Ex_n = 1}\Big(
  \sum_{k=1}^\infty \indic{X_{n}(k+1) = \circ;\Sw_n^- =  k}
  -\Sw_n^+\Big). \label{sprocess}
   \end {align}
Letting $\tau$ denote the minimal $n$ for which 
$\hat S_n\leq0$, we have that 
$S_n=\hat S_{n\wedge\tau}\vee0$.  That is, $S_n$ is obtained
by running $\hat S_n$ until it hits $\{\dotsc,-2,-1,0\}$
and then freezing it at 0.

Let $V_n=(\Ex_n,\Sw^-_n,\zeta_n,\chi_n)\in\ZZ^4$.
Note that $(V_n:n\geq0)$ is an i.i.d.\ sequence, and that $V_n$ is
independent of $X_n$ for each $n\geq0$.
We can write
\begin {equation}
 \hat{S}_{n+1} = \hat{S}_{n} +
 F\big(X_n,V_n\big) - \Sw^+_n
\end {equation}
where 
\[
F(X_n,V_n)=\indic{\Ex_n=3}\zeta_n+\indic{\Ex_n=2}\chi_n+
(\indic{\Ex_n=1}+\indic{\Ex_n=2}\chi_n)
\sum_{k=1}^\infty\indic{X_{n}(k+1) = \circ;\Sw_n^- = k}.
\] 
Thus
\begin {equation}\label{Sn_avg}
 n^{-1} \hat{S}_n = n^{-1}+ n^{-1} \sum_{i=0}^{n-1}
 F\big(X_i,V_i\big)
-n^{-1}\sum_{i=0}^{n-1} \Sw^+_i.
\end {equation}
Note that $F(X_n,V_n)\leq4$ is bounded.  
Also note that, in the positive-recurrent case, we may equivalently
write $F(X_n,V_n)=F^\ast(X^\ast_n,V_n)$ where 
\[
F^\ast(X^\ast_n,V_n)=\indic{\Ex_n=3}\zeta_n+\indic{\Ex_n=2}\chi_n+
(\indic{\Ex_n=1}+\indic{\Ex_n=2}\chi_n)
\sum_{k=1}^{W_n}\indic{X^\ast_{n}(k+1) = \circ;\Sw_n^- = k}.
\] 
Thus, applying Lemma~\ref{ergodic_lem} as well as the strong law of
large numbers to~\eqref{Sn_avg}, we deduce that
for all $p$
\begin{equation}\label{alpha_limit}
n^{-1}\hat{S}_n \asto \E[F(\xi^{(p)},V)]-\E(\Sw^+)
=\a(p).
\end{equation}

From~\eqref{alpha_limit} we see that
if $\a(p)<0$ then $\hat S_n\asto-\oo$, meaning that 
$S_n\asto0$.  Using Proposition~\ref{pc_prop}
it follows that 
\[
p_\crit\geq\sup\{p\in[0,1]:\a(p)<0\}.
\]
On the other hand, if $\a(p)>0$ then $\hat S_n\asto\oo$.
We claim that this implies that 
$\PP(S_n>0\mbox{ for all }n)>0$ 
and hence using
Proposition~\ref{pc_prop} again that
\[
p_\crit\leq\inf\{p\in[0,1]:\a(p)>0\}.
\]
To see the claim, first note that there is some $N$ such that
\[
\PP(\hat S_n>0\mbox{ for all }n\geq N)>0.
\]
Also recall that the process $((X_n,\hat S_n))_{n\geq0}$
is a Markov chain.
Fix a sample of $X_0$ and a sequence of peeling moves
$\pi=(\pi_0,\pi_1,\dotsc,\pi_{N-1},\pi_N,\dotsc)$
such that $\hat S_n>0$ for all $n\geq N$
(the $\pi_j$ encode which face is discovered and what the colours of
the new vertices are).  
We show that there are peeling
moves $\pi_0',\pi_1',\dotsc,\pi_{N-1}'$ such that if we instead
perform the sequence 
$\pi'=(\pi_0',\pi_1',\dotsc,\pi_{N-1}',\pi_N,\pi_{N+1},\dotsc)$
then (i) for all $n\leq N$,  the $X_n$ are the same as if we 
had performed the sequence $\pi$, (ii) 
$\hat S_{n+1}-\hat S_n\geq 0$ for all $n\leq N-1$, and (iii) $\hat
S_N$ is at least as large as if we had performed the  
sequence $\pi$.  Moreover, the
$\pi_n'$ can be chosen so that all the ratios $\PP(\pi_n')/\PP(\pi_n)$
are uniformly bounded from below by a positive number, for all choices
of the $\pi_n$.  Once we show that there are such $\pi'_n$ the claim
follows.  

We describe how to choose $\pi'_n$ given $\pi_n$ in a case-by-case
manner.  If $\Sw_n^+=0$ then we just take $\pi_n'=\pi_n$.  Assume in
what follows that $\Sw_n^+>0$, and to start with 
also that $\Sw_n^-=0$.  If $\pi_n$ exposes no vertex, or exposes exactly
one vertex which is white, then $\pi_n$ does not change $X_n$.
In this case let $\pi'_n$ be given by exposing two
vertices, both white.  Then $\pi'_n$ also does not change $X_n$, and  
$\PP(\pi'_n)/\PP(\pi_n)\geq\PP(\pi'_n)=\tfrac38 p^2$.
The next case is that $\pi_n$ exposes exactly one vertex which is
black, so that $X_{n+1}$ is obtained by inputting one black vertex to
the front of $X_n$.  We then let $\pi'_n$ be given by exposing two
vertices, the first white and the second black (counting
counter-clockwise).  Now
$\PP(\pi'_n)/\PP(\pi_n)\geq\PP(\pi'_n)=\tfrac38 p(1-p)$.
The final case is when $\pi_n$ is given by swallowing $k$ edges to the
left and $\ell$ edges to the right.  Then let $\pi'_n$ be given by
swallowing $k$ vertices to the left only, and exposing one white
vertex.  In this case we have, using~\eqref{qs_eq},
\be
\frac{\PP(\pi'_n)}{\PP(\pi_n)}=\frac{p q_k}{q_{k,\ell}}=
\frac{p q_k}{\tfrac83 q_k q_\ell}\geq \tfrac38 p.
\ee
This proves the claim and hence the theorem.
\qed

\subsection{Upper and lower bounds}

Although we are unable to explicitly find the
probabilites $\P(\xi^{(p)}(k)=\w)$, 
and hence the function $\a(p)$,  we
can find upper and lower bounds. 
As a warmup, we prove Proposition~\ref{bounds_prop}.
Trivially
$\beta(p)\geq0$ for all $p$, and this already gives
$p_\crit \leq\frac{\sqrt{73}-5}{6}$.
Moreover, from Lemma~\ref{iid_lem} we have that 
$\PP(\xi^{(p)}(k)=\w)\leq p$
for all $k$ (since the event $\{\xi(k)=\w\}$ is increasing).  
From this and the fact~\eqref{q-sums} 
that $\sum_{k\geq1,\mathrm{odd}}q_k=\tfrac{1}{8}$ and
$\sum_{k\geq1,\mathrm{even}}q'_k=\tfrac{1}{18}$
we deduce that
$p_\crit\geq \frac{\sqrt{493}-13}{18}$.
This proves Proposition~\ref{bounds_prop}.

We will now define a Markov chain with finite state space that will
allow us to improve these bounds.  To do this, it helps to first
recall the process 
$X^\ast_n=(X^\ast_n(k):1\leq k\leq W_n)\in\hat\S^\ast$, in particular
the fact that $X^\ast_n(k)=X_n(k)$ for all $k\leq W_n$.
Thus the `flag' $W_n$ keeps track of where in $X_n$ we can find
$X^\ast_n$.  For each $K\geq1$, we will define a process 
\[
X^K_n=(X^K_n(k):1\leq k\leq W^K_n)\in\hat\S^K.
\]
Here $\hat\S^K$ is the set of sequences in 
$\hat\S^\ast$ of length at most
$K$ and $W^K_n$ is the length of $X^K_n$.   We start with
$X^K_0=\b$, and thus $W^K_0=1$.
Supposing we have defined $X_n^K$ and $X_n$
for some $n\geq0$, we look at the next
peeling move of $X_n$.  If we `swallow beyond $X^K_n$', that
is $\Sw^-\geq W^K_n-1$, set $X^K_{n+1}=\b$.  Otherwise we first apply
the usual rules to $X^K_n$, and then (if necessary) truncate at $K$
to obtain $X^K_{n+1}$ which satisfies $W^K_{n+1}\leq K$. 
Since $X^K_{n+1}$ depends on $X^K_n$ and the independent randomness in
the next peeling move of $X_n$, it follows that $(X^K_n)_{n\geq0}$
is a Markov chain.  Moreover, we have coupled $X^K$ with $X$ and
$X^\ast$ in such a way that  for all $n$,
\begin{equation}\label{XK_eq}
W^K_n\leq W_n\wedge K,\mbox{ and }
X^K_n(k)=X^\ast_n(k)=X_n(k)\mbox{ for all }k\leq W^K_n.
\end{equation}
It is not hard to see that $X^K$ is an aperiodic and irreducible
Markov chain, and thus has a unique asymptotic distribution $\mu^K$.
The transition probabilities for this chain may be written down
explicitly, and hence also (at least in principle) the
measure $\mu^K$.

Define $X^{K,\b}_n$ by concatenating $X^K_n$ with an infinite sequence of
$\b$'s.  
If we start $X$ in the all-black state, then~\eqref{XK_eq}
implies that $X_n\geq X^{K,\b}_n$ for all $n$.  
Moreover, the distribution of $X^{K,\b}_n$
converges weakly, as $n\to \infty$, to a measure 
$\mu^{K,\b}$ which may be obtained from
$\mu^K$ in a straightforward way, and which satisfies
$\mu^{K,\b}\leq\mu^{(p)}$.  Thus we may use $\mu^{K,\b}$ to obtain a lower
bound on $\a(p)$ and hence an upper bound on $p_\crit$.
We may similarly define a process $X_n^{K,\w}$ by concatenating $X^K_n$ with
an infinite sequence of $\w$'s, and thus obtain a measure $\mu^{K,\w}$
satisfying $\mu^{(p)}\leq\mu^{K,\w}$.  However, there is another, better,
way to obtain an upper bound on $\mu^{(p)}$, as follows.  

Recall from Lemma~\ref{iid_lem} that we coupled $X$ to a chain $Y$
such that for each $n$ the distribution of $Y_n$ is {\iid}($p$),
and $X_n\leq Y_n$.  
Define $X^{K,\iid(p)}_n(k)$ to be
$X_n(k)=X^K_n(k)$ if $k\leq W^K_n$, or $Y_n(k)$ otherwise.  Thus we
have for all $n$ that $X^{K,\iid(p)}_n\geq X_n$.  
Moreover, the distribution of $X^{K,\iid(p)}_n$
converges weakly, as $n\to \infty$, to a measure 
$\mu^{K,\iid(p)}$ which may be obtained by
first sampling $\xi^K\in\hat\S^K$ from $\mu^K$ and then appending to
it an infinite {\iid}($p$) sequence
(independent of $\xi^K$).  It follows that $\mu^{K,\iid(p)}$ 
stochastically dominates $\mu^{(p)}$.

As an example, taking $K=2$ the relevant states of $X^{2}$ 
are $\b$, $\b\b$ and $\w\b$ (each state read from right to left).
Using~\eqref{q-sums}--\eqref{1k} we find that
the transition probabilities are:
\begin{equation}
\begin{array}{l| r}
\mathrm{Transition} & \mathrm{Probability} \\
\hline 
\b \to \b\b &  \frac{1}{2}(1-p)\\ [2pt]
\b \to \w \b & \frac{3}{8} p (1-p)\\ [2pt]
\b\b \to \b & \frac{1}{9}(1+p)\\ [2pt]
\b\b \to \w \b & \frac{3}{8} p (1-p) \\  [2pt]
\w\b \to \b & \frac{1}{9}(1+p)\\ [2pt]
\w\b \to \b\b & \frac{1}{2}(1-p)\\
\end{array}
\end{equation}
and hence 
\begin {align}
 \mu^2(\b) &= \frac{8(1+p)}{-27 p^2 - p + 44}\\
 \mu^2(\w\b) &= \frac{27p(1-p)}{-27 p^2 - p + 44}.
\end {align}
Using that $\mu^{2,\b}(\w \b) = \mu^2(\w \b)$ 
and $\mu^{2,\iid(p)}(\w \b)= p\mu^2(\b) + \mu^2(\w \b)$
as well as~\eqref{alpha-eq} and~\eqref{q-sums}
one finds that a lower bound on $p_\crit$ 
is given by the unique solution in $[0,1]$ to
\begin{equation}
189 p^4+378 p^3-596 p^2-575 p + 396 = 0
\end{equation}
and an upper bound is given by the unique solution in $[0,1]$ to
\begin{equation}
81 p^4+162 p^3-251 p^2-232 p + 176 = 0.
\end{equation}
The result is 
$0.523599 \leq p_\crit\leq 0.572542$ when rounding to 
six digits.

One may similarly write down the transition probabilities
for $X^K$ for general $K$, but
as $K$ becomes larger it quickly becomes infeasible to write down the
limiting distributions $\mu^{K,\b}$ and
$\mu^{K,\iid(p)}$ by hand.  We 
provide in 
Table~\ref{bounds_tab} some upper and lower bound 
which we numerically computed
using this method, and Fig.~\ref{f:plot} shows a plot
of these bounds for $K$ in the range 2 to 17.  
We note that the gap between the bounds we obtain is (weakly)
decreasing in $K$ due to stochastic monotonicity of the measures
$\mu^{K,\b}$  and $\mu^{K,\iid(p)}$ in $K$. 

\begin{figure}
\begin{floatrow}
\capbtabbox{%
  \begin{tabular}{l|cc}
$K$ & Lower bound & Upper bound \\
\hline 
4  & 0.5382 & 0.5656 \\
6  & 0.5436 & 0.5625 \\
8  & 0.5464 & 0.5609 \\
10 & 0.5482 & 0.5598 \\
12 & 0.5493 & 0.5591 \\
14 & 0.5502 & 0.5586 \\
16 & 0.5508 & 0.5583 \\
17 & 0.5511 & 0.5581 \\
\end {tabular}
}{\captionsetup{width=5cm}
  \caption{Upper and lower bounds on $p_\crit$, obtained by 
numerically finding the limiting distributions for the 
processes $X^{K,\iid(p)}_n$ and $X^{K,\b}$.}\label{bounds_tab}
}
\ffigbox{%
  \centerline{\scalebox{0.38}{\includegraphics{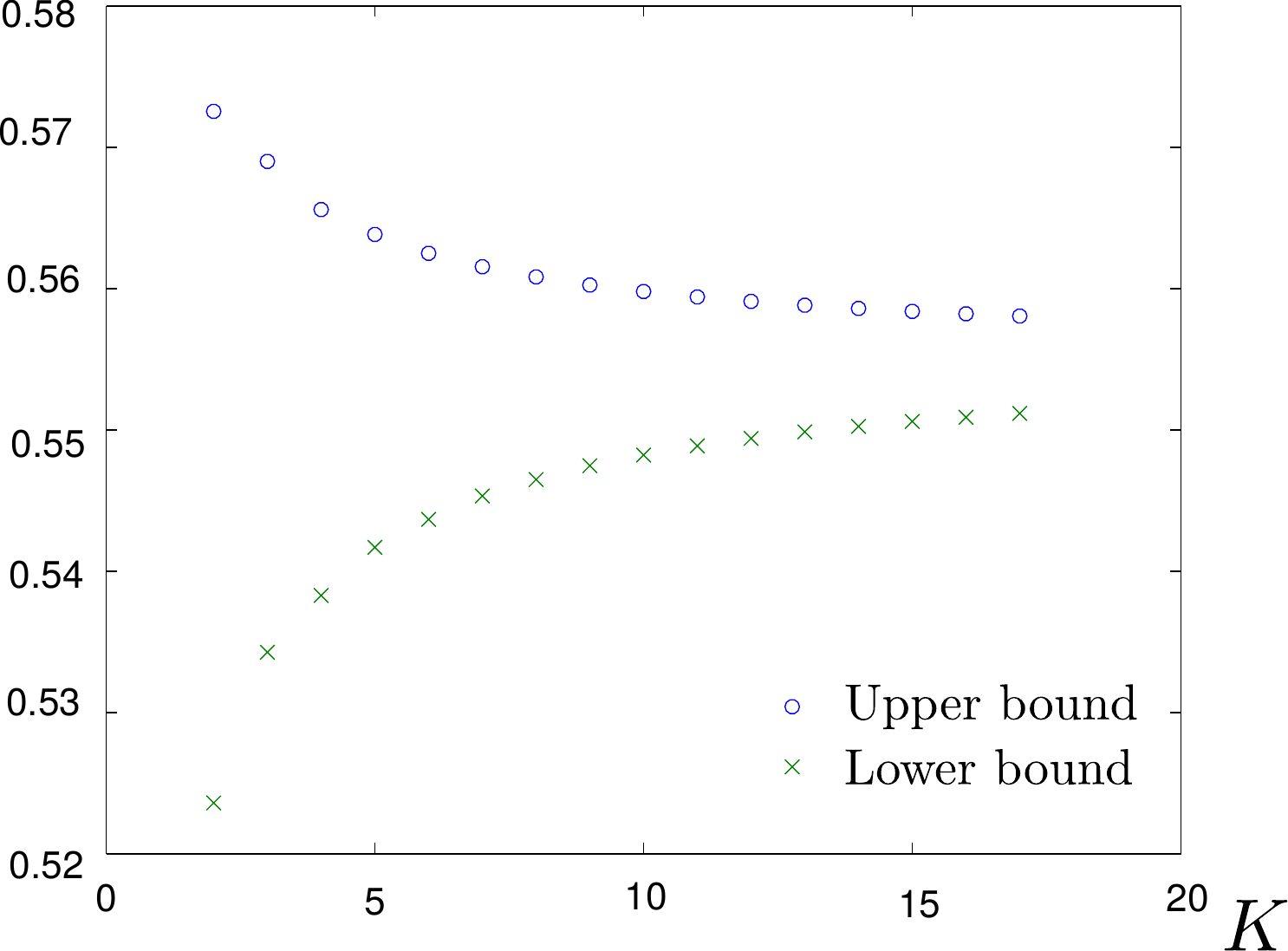}}}

}{\captionsetup{width=5cm}
  \caption{A plot of upper bounds and lower bounds on $p_c$ as a function of $K$.} \label{f:plot}
}
\end{floatrow}
\end{figure}

\section{Outlook}

Certain questions are left unanswered by this work.
It may be possible
to find the exact distribution of $\xi^{(p)}$, and thereby
recover Richier's result that $p_\crit=\tfrac59$
(see Remark~\ref{richier-rk})
but this would require a new idea.  We also do not provide much
information about general properties of the function $\a(p)$, for
example we have not showed that it has a \emph{unique} root in
$[0,1]$, which seems natural to suppose.  A related question is
whether the processes $(X_n)_{n\geq0}$ and $(S_n)_{n\geq0}$ are
stochastically monotonic in $p$?  This also does not seem easy to
establish.  

On the other hand, the methods we have presented should (at least in
principle) not be hard to 
extend to site percolation on other half-planar maps with the
domain Markov property.  Angel and Ray showed 
in~\cite{angel-ray} that
for each $k\geq3$ there is a one-parameter family of 
translation-invariant, domain Markov probability measures
supported on half-planar
$k$-angulations with only simple faces.
Extending the methods to the full class of such 
\emph{quadrangulations} appears straightforward:
one need only adjust the values of the $q$:s.  The methods should in
principle also extend to $k\geq5$.  Then the notion of admissible boundary
conditions would need to be modified to allow for segments of
consecutive white vertices of length up to $k-3$, and the formulas
would become considerably more complicated, but there does not seem to
be any fundamental problem.

\subsection*{Acknowlegdement}

This work was started while both authors were at Uppsala
University in Sweden.  We have benefited from discussions 
with Svante Janson, Takis Konstantopoulos, Pierre Nolin
and Hermann Thorisson.


\end{document}